\documentclass[11pt,reqno]{amsart}

\usepackage{latexsym}
\usepackage{amsfonts}
\usepackage{amsmath}
\usepackage{url}
\usepackage{graphicx}
\usepackage{color}

\usepackage{epsfig,psfrag,verbatim,amsfonts}

\usepackage{amssymb,amsmath}

\def\bl{\begin{lemma}}
\def\el{\end{lemma}}
\def\bth{\begin{theorem}}
\def\eth{\end{theorem}}
\def\bc{\begin{corollary}}
\def\ec{\end{corollary}}
\def\bcj{\begin{conjecture}}
\def\ecj{\end{conjecture}}
\def\bpr{\begin{proposition}}
\def\epr{\end{proposition}}
\def\bde{\begin{definition}}
\def\ede{\end{definition}}
\def\E{\mathbb{E}}

\def\Pr{\mathbb{P}}

\newcommand{\be}{\begin{eqnarray}}
\newcommand{\ee}{\end{eqnarray}}

\newcommand{\R}{{\mathbb R}}

\renewcommand{\and}{\hbox{ {\rm and} }}

\newtheorem{theorem}{Theorem}[section]
\newtheorem{definition}{Definition}[section]
\newtheorem{lemma}[theorem]{Lemma}

\newtheorem{corollary}[theorem]{Corollary}
\newtheorem{proposition}[theorem]{Proposition}
\newtheorem{conjecture}[theorem]{Conjecture}
\newtheorem{question}[theorem]{Question}

\theoremstyle{definition}
\numberwithin{equation}{section}

\begin{document}
\title{Where does a random process hit a fractal barrier?}
\author{Itai Benjamini \and Alexander Shamov}

\date{April 2016}

\maketitle

\begin{abstract}
Given a Brownian path $\beta(t)$ on $\R$, starting at $1$, a.s. there is a singular time set $T_{\beta}$, such that   the first hitting time of $\beta$ by an independent
Brownian motion, starting at $0$, is in $T_{\beta}$ with probability one.
A couple of problems regarding hitting measure for  random processes are presented.
\end{abstract}

\section{Introduction}

The study of Harmonic (or hitting) measure for Brownian motion is a well developed subject with dramatic achievements and major problems which are still wide open, see ~\cite{GM}.
In this note we present a couple of problems regarding hitting measure for a wider class of random processes and obtain one result.

When does one dimensional Brownian motion starting at $0$, hits an independent Brownian motion  starting at $1$, which serves as the {\em barrier}?

We show that conditioning on the  barrier, a.s. with respect to the Wiener measure on barriers,
there is a singular time set (which is a function of the barrier only) that a.s. contains  the first hitting time of the barrier.

\section{Random processes in the plane}

Let $\gamma$ be an unbounded  one sided curve in the Euclidean plane.  Given a simply connected open bounded  domain $\Omega$ in the plane.
Reroot the origin of $\gamma$ at a uniformly chosen point of $\Omega$, and rotate $\gamma$ with an independent uniformly chosen angel, around it's root.
Look at the hitting point of this random translation and rotation of $\gamma$ on the boundary of the domain $\partial \Omega$.
For every root in $\Omega$ the hitting point maps the uniform measure on directions  $U[0,2\pi]$ to a measure on  $\partial \Omega$.

\begin{conjecture}
For any $\gamma$ and $\Omega$, for almost every root, the corresponding measure on  $\partial \Omega$ has $0$  two dimensional Lebesgue measure.
\end{conjecture}

Moreover,

\begin{question}
For any $\gamma$ and $\Omega$, for almost every root, the corresponding measure on  $\partial \Omega$ has Hausdorff dimension (at most) one?
\end{question}

It is of interest to prove even that dimension drops below $2$. Or better below the dimension of $\partial \Omega$ when it is strictly above $1$.
Also getting the result for a restricted family of curves, is of interest.
If $\gamma$ is a Brownian path then Makarov's theorem \cite{Ma} gives an affirmative answer.
For partial results on this conjecture when $\gamma$ is a straight line see ~\cite{FF}.

\subsection{Simple random walks on discrete fractals}

By Makarov's theorem~\cite{Ma} (and Jones and Wolff ~\cite{JW} for general domains) and it's adaptation by Lawler~\cite{La} via coupling to simple random walk,
it is  know that the dimension of the hitting measure for two dimensional Brownian motion drops to (at most) $1$.
We therefore suspect that harmonic measure for simple random walk on self similar planar fractals will also be at most $1$.
Here is a specific formulation.

\subsubsection{Sierpinski gasket}

\begin{figure}

    \centering
    \includegraphics[width=0.8\textwidth,natwidth=610,natheight=120]{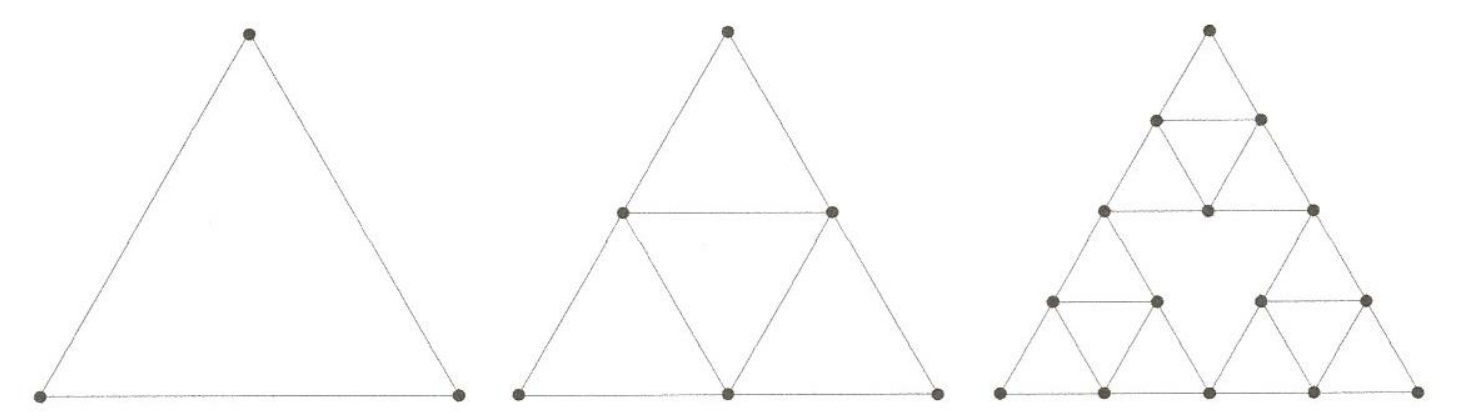}
    \caption{The first three generation of  Sierpinski gasket graphs sequence. \label{fig:S-gasket}}
\end{figure}

Given a subset $S$ of the vertices in the $n$-th generation of the  Sierpinski gasket graph sequence (see Figure~\ref{fig:S-gasket}).
\begin{question}
Show that the entropy of the hitting  measure  for a simple random walk starting at the top vertex on $S$ is at most $n$.
\end{question}

Note that in the $n$-th generation Sierpinski gasket graph, the size of the bottom side is $2^{n-1}+1$, which we believe
realizes the largest entropy possible.  (Entropy in base $2$,    $- \sum_i p_i \log_2 p_i$).

\subsection{Fractional BM}

Recall the probability Brownian motion in $\R^2$, starting at $(1,0)$   hits the negative $x$-axis first at $[-\epsilon, 0]$ behaves like $\epsilon^{1/2}$,
as epsilon goes to $0$.

We would like to have a natural statement along the lines that the rougher the process starting at $(1,0)$
the  larger the probability it will hit the negative $x$-axis first  near the tip.
E.g. if the process starting at $(1,0)$ is a two dimensional fBM with Hurst parameter $H$,
then as $H$ decreases the probability it hits the $\epsilon$ tip is growing (maybe it is about $\epsilon^H$ ?)

One can ask similar a question for the graph of one dimensional fBM and SLE curves.

\section{Random process on the line}

\begin{theorem} \label{thm:Singularity}
 Let $B$ and $W$ be independent standard Brownian motions on $\R$, and let $\sigma, c > 0$. Define $\tau$ to be the first time when $B$
 hits the barrier $c + \sigma W$, i.e.
 $$\tau := \inf\{t \mid B_t = c + \sigma W_t\}.$$
 Then conditionally on $W$, the distribution of $\tau$ is almost surely singular to the Lebesgue measure.
\end{theorem}

In the proof we will make use of the following standard fact from measure theory.

\begin{proposition} \label{prop:Disintegration}
 Let $M, N$ be probability measures on $X \times Y$, a product of standard Borel spaces. Consider the disintegration of $M, N$ with
 respect to the $X$-variable (i.e. with respect to the canonical projection $X \times Y \to X$). We write it as follows:
 $$M(dx, dy) = M_X(dx) M_{Y \mid X}(x, dy)$$
 $$N(dx, dy) = N_X(dx) N_{Y \mid X}(x, dy)$$
 where $M_X$ (resp. $N_X$) is the pushforward of $M$ (resp $N$) under $X \times Y \to X$, and $M_{Y \mid X}(x)$
 (resp. $N_{Y \mid X}(x)$) is corresponding conditional of $y$ given $x$.
 Assume that $M_X$ is equivalent (i.e. mutually absolutely continuous) to $N_X$. Then the following are equivalent:
 \begin{enumerate}
  \item $M$ is singular to $N$
  \item $M_{Y \mid X}(x)$ is singular to $N_{Y \mid X}(x)$ for $M_X$-almost all $x$
 \end{enumerate}
\end{proposition}

%
%

Another fact we will need is the Bessel(3)-like behavior of the Brownian motion immediately before hitting a constant barrier.
This is an immediate consequence of Williams' Brownian path decomposition theorem (e.g. Theorem VII.4.9 in \cite{RY}).

\begin{proposition} \label{prop:Bessel3}
 Consider a Brownian motion $B$ starting from $0$, and let $c>0$. Let $\tau$ be the hitting time $\tau := \inf{t \mid B_t = c}$. Then
 for any $\varepsilon > 0$ the conditional distribution of $(c - B_{T-s})_{s=0}^{T-\varepsilon}$ conditioned on
 $\tau = T > \varepsilon$ is equivalent to that of a Bessel(3) process starting from $0$ restricted to the time interval
 $[0,T-\varepsilon]$.
\end{proposition}

\begin{proof}[Proof of Theorem \ref{thm:Singularity}]
 Consider the random measure $\delta_\tau$ and $D := \E[\delta_\tau \mid W]$. The latter is exactly the conditional
 distribution of $\tau$ given $W$. By Proposition \ref{prop:Disintegration} (applied to $X := \Omega$, $Y := \R$),
 almost sure singularity of $D = D(\omega, dt)$ to the (determinstic) Lebesgue measure is equivalent to the singularity of
 $\Pr(d\omega) D(\omega, dt)$ to $\Pr(d\omega) dt$. On the other hand, the two spaces $X$ and $Y$ in Proposition
 \ref{prop:Disintegration} play symmetric roles, so instead one may disintegrate with respect to the $t \in \R$ variable.
 More precisely, let $\Pi(t, d\omega)$ (resp. $\pi(t, d\omega)$) be the disintegration of $\Pr(d\omega) D(\omega, dt)$
 (resp. $\Pr(d\omega) \delta_\tau(\omega, dt)$) with respect to $t$. Then the $\Pr$-almost sure singularity of $D$ with respect to the
 Lebesgue measure is equivalent to the singularity of $\Pi(t)$ with respect to $\Pr$ for Lebesgue-almost all $t$. On the other hand, the
 measures $\Pr(d\omega) D(\omega, dt)$ and $\Pr(d\omega) \delta_\tau(\omega, dt)$ agree when restricted to the $\sigma$-algebra
 $\sigma(W) \otimes \mathrm{Borel}(\R)$; therefore, $\Pi(t)$ and $\pi(t)$ agree on $\sigma(W)$ for Lebesgue-almost all $t$. Since $D$ is
 measurable with respect to $\sigma(W)$, it is enough to verify that $\pi(t)$ is singular to $\Pr$ when restricted to $\sigma(W)$.

 Using Proposition \ref{prop:Bessel3} we can characterize $\pi(t)$ explicitly, at least up to equivalence. Indeed, the time when $B$
 hits $c + \sigma W$ is exactly the time when
 $$X := \frac{1}{\sqrt{1 + \sigma^2}} B - \frac{\sigma}{\sqrt{1 + \sigma^2}} W,$$
 which is itself a standard Brownian motion under $\Pr$, hits the constant barrier $\tilde c := \frac{c}{\sqrt{1 + \sigma^2}}$. Thus by
 Proposition \ref{prop:Bessel3}, the distribution of $\tilde c - X_{t - \cdot}$ under $\pi(t)$ is (locally) equivalent to Bessel(3).
 On the other hand,
 $$Y := \frac{\sigma}{\sqrt{1 + \sigma^2}} B + \frac{1}{\sqrt{1 + \sigma^2}} W$$
 is $\Pr$-independent of $X$, and since the $\tau$ is measurable with respect to $X$, the independent part $Y$ is not affected by our
 change of measure. Thus under $\pi(t)$, $X$ and $Y$ are still independent, and $Y_{t - \cdot}$ remains (locally) equivalent to a
 Brownian motion.

 In order to prove the singularity result we only need the restriction of our measures to $\sigma(W)$. Since
 $$W = -\frac{\sigma}{\sqrt{1+\sigma^2}} X + \frac{1}{\sqrt{1+\sigma^2}} Y,$$
 we see that under $\pi(t)$, $W_{t - \cdot}$ is locally equivalent to a combination of a Bessel(3) and an independent Brownian
 motion. Under $\Pi$, however, it is locally a Brownian motion. Thus the problem reduces to the proving that
 the local behaviour at time zero of the sum of independent processes
 $$U \sim \alpha \cdot \mathrm{BES(3)} + \sqrt{1 - \alpha^2} \cdot \mathrm{BM}$$
 is almost surely distinguishable from that of $V \sim \mathrm{BM}$, where $\alpha = -\frac{\sigma}{\sqrt{1+\sigma^2}} < 0$.
 This can be achieved by, say, noting that these processes satisfy a law of iterated logarithm with different almost sure constants.
 Namely,
 $$\limsup_{s \to 0} \frac{V_s}{\sqrt{2 s \log \log s}} = 1$$
 $$\limsup_{s \to 0} \frac{U_s}{\sqrt{2 s \log \log s}} \le \sqrt{1 - \alpha^2} < 1$$
\end{proof}

\begin{question}
Study this phenomena for larger class of barriers, e.g. iterated function systems.
Give sharper bounds on the dimension of the the set which a.s. contains the hitting time.
\end{question}

To study this for iterated function systems, we need a  uniform bound on the radon nikodym
derivative of the harmonic measure with respect to the uniform measure, at all scales.
\medskip

Here is a formulation of this problem for {\em random fields}. Consider a function from $\R^d$ to $\R^n$ as a barrier, and look when a random field indexed by $\R^d$
hits the barrier, where the hitting index is defined say as the index with the smallest $L_2$ norm.

\section{Further comments}

\begin{itemize}

\item{\em Bourgain's proof}

Bourgain~\cite{Bo} proved a dimension drop result for Brownian motion in $\R^d$ for any $d$.
Two properties of BM are used in the clever argument, {\em uniform Harnack inequality}
at all scales and the {\em Markov property}, to get independent between scales.
These two properties hold for a wider set of processes in a larger set of spaces, (e.g. Brownian motion on  nilpotent groups and fractals).
Also weaker forms of these properties are sufficient to get some drop.

\item{\em Random walk on graphs}

This note concerns with harmonic measure in "small spaces" of dimension at most two.
See ~\cite{BY} for a study of hitting measure for the simple random walk in the presence of a spectral gap:
on highly connected graphs such as expanders, simple random walk is mixing fast and it is shown that it hits
the boundary of sets in a rather uniform way.
More involved behavior arises for graphs which are neither polynomial in the diameter nor expanders,
see ~\cite{BY}.

\item{\em Let's play}

Rules: each of the $k \geq 2 $ players  picks independently a unit length path (not necessarily a segment)  in the Euclidean plane that contains the origin.
Let $S$ be the union of all the $k$ paths. Look at the harmonic measure from infinity on $S$.
The winner is the player that his path, gets the maximal harmonic measure.

Is choosing a segment from the origin to a random point on the unit circle, independently by each of the  players, a {\em Nash equilibrium}?

\end{itemize}

\medskip
\noindent

\end{document}